\documentclass[leqno,12pt]{amsart}
\usepackage{amscd,amsfonts,amsmath,amssymb,amstext,amsthm}
\newtheorem{theorem}{Theorem}[section]
\newtheorem{lemma}[theorem]{Lemma}

\newtheorem{corollary}[theorem]{Corollary}
\begin{document}
\title{On the Commutator Map\\
for\\
Real Semisimple Lie Algebras}
\author{Dmitri Akhiezer}
\address
{Dmitri Akhiezer\newline
Institute for Information Transmission Problems\newline
19 B.Karetny per.,127994 Moscow, Russia}
\noindent
\email{akhiezer@iitp.ru}
\begin{abstract}
We find new sufficient conditions for the commutator map of a real 
semisimple Lie algebra
to be surjective. As an application, we prove the surjectivity
of the commutator map for all simple algebras except 
${\mathfrak s \mathfrak u}_{p,q}$ ($p$ or $q > 1$), 
${\mathfrak s \mathfrak o}_{p,p+2}$\ ($p$ odd or $p=2$),\
${\mathfrak u}^*_{2m+1}({\mathbb H})$\ ($m\ge 1$) and $EIII$. 
\end{abstract}
\renewcommand{\subjclassname}
{\textup{2010} Mathematics Subject Classification}
\maketitle
\renewcommand{\thefootnote}{}
\footnotetext{Supported by SFB/TR 12 and SPP 1388 of the DFG}
\maketitle

\section{Introduction and statement of results}\label{intro}
Let $\mathfrak g$ be a semisimple Lie algebra.
The commutator map 
$$\mathfrak g \times \mathfrak g\to \mathfrak g,\ (X,Y)\mapsto [X,Y],$$
is known to be surjective if $\mathfrak g$
is split. This result, due to Gordon Brown\,\cite{Br},
is valid for all infinite fields and for finite fields of sufficiently big cardinality.  
If $\mathfrak g $ is any real semisimple Lie algebra
then its complexification
$\mathfrak g^{\mathbb C}$ is split over $\mathbb C$.
Therefore, for any $Z \in {\mathfrak g}$ there exist
$X_1, X_2, Y_1, Y_2 \in \mathfrak g$, such that
$Z = [X_1+iX_2, Y_1+iY_2]$. Hence 
$Z$ is the sum of two commutators, namely, $Z = [X_1, Y_1] - [X_2, Y_2]$.
To the best of author's knowledge, 
the surjectivity of the
commutator map is in general not established. On the other hand,  
there is no example where two commutators in the above formula are indeed essential,
so that a presentation
of $Z$ as one commutator does not exist.

Our goal is to obtain new sufficient conditions for the surjectivity
of the commutator map.

\begin{theorem}\label{main}
Let $\mathfrak g$ be a semisimple real Lie algebra, $\mathfrak g = \mathfrak k + \mathfrak p$
a Cartan decomposition, 
$\mathfrak a $ a maximal abelian subspace in $\mathfrak p$, and
$\mathfrak m = {\mathfrak z}_{\mathfrak k}(\mathfrak a)$ the centralizer of $\mathfrak a$ in $\mathfrak k$.
If $\mathfrak m$ is semisimple
then the commutator map $\mathfrak g \times \mathfrak g \to \mathfrak g , (X,Y)\mapsto [X,Y],$ is 
surjective.
\end{theorem}

If $\mathfrak g$ is split then $\mathfrak a$ is a Cartan subalgebra
of $\mathfrak g$ and $\mathfrak m = \{0\}$. 
If $\mathfrak g$ is compact then $\mathfrak a$= \{0\} and $\mathfrak m = \mathfrak g$.
Therefore we have the following corollary.

\begin{corollary}\label{cor} If $\mathfrak g $ is split or compact then the commutator map is surjective.
\end{corollary}
As we noted above, the result in the split case is known, see\,\cite{Br}. In the compact case Corollary\,\ref{cor}  
is easily deduced from Theorem\,3.4 in\,\cite{DT}. Our proof in both cases is new.
It is not even necessary to invoke Theorem ~\ref{main} in full generality, 
see Corollaries~\ref{split} and ~\ref{compact}.

Clearly, the question of surjectivity of the commutator map
reduces to simple Lie algebras. Apart from real split, compact and complex simple algebras,
we have the following list, see e.g.\,\cite{On} for the notations.

\begin{theorem}\label{simple}
The commutator map is surjective for classical algebras
${\mathfrak s \mathfrak l}_m({\mathbb H}) (m \ge 2), {\mathfrak s \mathfrak o}_{p,q} 
(1\le p\le q, q \ne p + 2),
 {\mathfrak u}^*_{2m}{(\mathbb H)} (m\ge 2),
 {\mathfrak s \mathfrak p}_{p,q} (1\le p \le q)$ and for 
exceptional algebras EIV, EVI, EVII, EIX, FII.
\end{theorem}

Recall that $\mathfrak g$ is said to be of inner type 
if the Cartan involution of $\mathfrak g$ is an inner automorphism or,
equivalently, if $\mathfrak k$ has full rank. 
The list of simple 
algebras of inner type comprises all Hermitian algebras as well as
${\mathfrak s \mathfrak o}_{p,q}$ with $p$ or $q$ even, $p,q\ne 2$, 
${\mathfrak s \mathfrak p}_{p,q}$, and
exceptional algebras 
$EII, EV, EVI, EVIII, EIX, FI, FII,G$.  

\begin{theorem}\label{inner} If $\mathfrak g$ is a simple non-Hermitian algebra of inner type
then the commutator map of $\mathfrak g$ is surjective. 
In particular,
the commutator map is surjective for
${\mathfrak s \mathfrak o}_{p, p+2}$ ($p>2,\ p$ even) and $EII$.
\end{theorem}

To summarize, Corollary ~\ref{cor}, Theorem ~\ref{simple} and Theorem ~\ref{inner} show that the only real simple algebras,
for which the surjectivity of the commutator map is an open question, are
${\mathfrak s \mathfrak u}_{p,q}$ ($p$ or $q > 1$), 
${\mathfrak s \mathfrak o}_{p,p+2}$\ ($p$ odd or $p=2$),\
${\mathfrak u}^*_{2m+1}({\mathbb H})$\ ($m\ge 1$) and $EIII$, 
with well-known overlaps between series. 

The author is grateful to D.I.Panyushev for useful discussions.
\section{Preliminaries}\label{prel}
We gather here several known facts which will be used later on.
Let ${\mathfrak z}(A)$ denote the centralizer
of an element $A$ in a Lie algebra $\mathfrak g$.  

\begin {lemma}\label{ortho}
Assume $\mathfrak g$ has a non-degenerate bilinear form $\beta $. Then
$${\rm Im}\ {\rm ad}(A)= {\mathfrak z}(A)^\perp ,$$
where the orthogonal complement is taken with respect to any such form $\beta $.
In particular, this complement is independent of $\beta $.
\end{lemma} 
\begin{proof}
For any $B \in {\mathfrak z}(A)$ and any $X \in {\mathfrak g}$ one has
$$\beta (B, [A,X]) = \beta ([B,A], X) = 0,$$
showing that ${\mathfrak z}(A) \subset {\rm Im}\ {\rm ad }(A)^\perp $.
Since $\beta $ is non-degenerate, this inclusion is in fact
an equality by the dimension argument.
\end{proof} 

\begin{lemma} \label{repr}
Let $V$ be a real vector space and $\Gamma \subset {\rm GL}(V)$
a finite linear group acting without fixed vectors. Then
the convex hull of any $\Gamma $-orbit contains $0\in V$.
\end{lemma}
\begin{proof}
Assigning mass 1 to every point of an orbit $\Gamma \cdot v$, we see that the mass center
is a fixed vector in the convex hull ${\rm Conv}(\Gamma \cdot v)$. Hence $0\in {\rm 
Conv}(\Gamma\cdot v)$.
\end{proof}
Let $\mathfrak g = \mathfrak k + \mathfrak p$ be a Cartan decomposition of a real
semisimple Lie algebra $\mathfrak g$, let $\mathfrak a \subset \mathfrak p$ be a Cartan
subspace, and let $K$ be the adjoint group ${\rm exp} ({\rm ad}\, {\mathfrak k}) \subset {\rm GL}({\mathfrak g})$.
The Weyl group acting on $\mathfrak a$
is denoted by $W$. 
The following result is due to B.Kostant \cite{Ko}.
\begin{lemma}\label {conv}
Consider the orthogonal decomposition $\mathfrak p = \mathfrak a + \mathfrak a^\perp$
with respect to the Killing form and let $\pi : {\mathfrak p} \to {\mathfrak a}$
be the corresponding projection map. Then for any $X\in {\mathfrak p}$ 
the intersection $K\cdot X \cap {\mathfrak a}$ is a single $W$-orbit and one has
$$\pi (K\cdot X) =  {\rm conv}(K\cdot X \cap \,{\mathfrak a}).$$
\end{lemma}
\begin{proof} The proof is found in \cite{Ko}, see Prop.\,2.4 for the first statement and Thm.\,8.2 for the second one.
\end{proof}
\begin{corollary}\label{Cartan}
For any $X \in {\mathfrak p}$ there exists an element $k \in K$, such that $k\cdot X \in {\mathfrak a}^\perp$.
\end{corollary}
\begin{proof}
The convex hull of a Weyl group orbit contains 0 by Lemma \ref{repr}. Therefore, by Lemma \ref{conv}
there exists $k \in K$, such that $\pi (k\cdot X) = 0$. 
\end{proof}
\section{Commutator map ${\mathfrak k}\times {\mathfrak p} \to 
{\mathfrak p}$}
We keep the above notations. Recall that an element $A\in {\mathfrak p}$ is called regular if
${\mathfrak z}(A) \cap {\mathfrak p}$ is a Cartan subspace.
\begin{theorem}\label{mixed}
For any $X \in {\mathfrak p}$ there exist $Y \in {\mathfrak k}$ and a regular $A \in {\mathfrak p}$,
such that $X = [Y,A]$.
\end{theorem}
\begin{proof} By Corollary ~\ref{Cartan} 
we may assume that $X \in {\mathfrak a}^\perp$. Take any
regular element $A\in {\mathfrak a}$. 
Its centralizer is of the form 
${\mathfrak z}(A) = {\mathfrak z}(A)\cap {\mathfrak k} + 
{\mathfrak a }$. 
Since $\mathfrak k$ and $\mathfrak p$ 
are orthogonal with respect to the Killing form,
we have 
$X \in {\mathfrak z}(A)^\perp$. 
Therefore $X = [Z,A]$ for some $Z \in {\mathfrak g}$ by Lemma ~\ref{ortho}.
Now write $Z$ as the sum $Z = Y + Y^\prime$, where $Y \in {\mathfrak k}, Y^\prime \in {\mathfrak p}$. 
Then $[Y^\prime, A] \in {\mathfrak k}$ and, in fact, $[Y^\prime, A] = 0$, hence $X = [Y,A]$.
\end{proof}
\begin{corollary}\label{split}
If $\mathfrak g$ is split then for any $X \in {\mathfrak g}$ there exist $Y \in {\mathfrak g}$ and a regular $A
\in {\mathfrak p}$, such that $X = [Y,A]$.
\end{corollary}
\begin{proof}
Let $X = X_1 + X_2$, where $X_1 \in \mathfrak k,\,X_2 \in \mathfrak p$. 
As in the proof of Theorem~\ref{mixed}, we may assume that $X_2 \in {\mathfrak a}^\perp$. But $X_1$ is also
in ${\mathfrak a}^\perp$, hence $X \in {\mathfrak a}^\perp$. Since $\mathfrak g$ is split,
$\mathfrak a$ is a Cartan subalgebra in $\mathfrak g$. Thus $\mathfrak a = {\mathfrak z}(A)$
for any regular $A \in {\mathfrak a}$, and
our assertion follows from Lemma ~\ref{ortho}. 
\end{proof}
\begin{corollary}\label{compact}
Any element of a compact semisimple Lie algebra is a commutator of two elements and one of these two can be chosen regular.
\end{corollary}
\begin{proof} Let $\mathfrak k$ be a compact semisimple subalgebra. We can apply Theorem~\ref{mixed} to the algebra
$\mathfrak g = \mathfrak k^{\mathbb C} = {\mathfrak k} + i{\mathfrak k}$. Here $\mathfrak p = i\mathfrak k$,
and so for any $X\in \mathfrak k$ we get a presentation $iX = [Y, iA]$ or, equivalently, $X = [Y, A]$,
where $Y, A \in \mathfrak k$ and $A$ is regular.
\end{proof}

\section{Proofs of main results}
We shall prove a theorem which is slightly more precise than Theorem~\ref{main}.
Recall that $\mathfrak g = \mathfrak k+ \mathfrak p$ is the Cartan decomposition
of a semisimple Lie algebra $\mathfrak g$ and $\mathfrak m$ is the centralizer
of a Cartan subspace $\mathfrak a \subset \mathfrak p$ in $\mathfrak k$. 
Choose a Cartan subalgebra ${\mathfrak t} $ in the reductive algebra  ${\mathfrak m}$.
The orthogonal complement to a subspace $V \subset \mathfrak g $ in $\mathfrak g$
with respect to the Killing form is denoted by $V^\perp$.
The Killing form is positive definite on $\mathfrak p$, negative definite
on $\mathfrak k$, and $\mathfrak p$ and $\mathfrak k$ are mutually orthogonal.
For $V \subset \mathfrak k$, resp.
$V \subset \mathfrak p$, we write $V^\perp_{\mathfrak k} $ instead of $V^\perp \cap {\mathfrak k}$,
resp $V^\perp_{\mathfrak p}$ instead of $V^\perp \cap {\mathfrak p}$. 
\begin{theorem}\label{main1}
If $\mathfrak m $ is semisimple then any element of $\mathfrak g$
is contained in the image of ${\rm ad}\, (A)$ for some regular element $A \in {\mathfrak g}$.
\end{theorem}
\begin{proof}
Let $M ={\rm exp}\,({\rm ad}\, {\mathfrak m}) \subset K$. 
We can apply Corollary~\ref{Cartan} to the semisimple algebra $\mathfrak m+ i\mathfrak m$ and the the Cartan
subspace $i\mathfrak t \subset i\mathfrak m$. Thus, given $Y \in {\mathfrak m}$ we can find $m \in M$ such that
$m\cdot Y $ is orthogonal to $\mathfrak t$ with respect
to the Killing form of $\mathfrak m$. However, the orthogonal complement of $\mathfrak t$ in $\mathfrak m$ is the same
for all invariant non-degenerate 
bilinear forms on $\mathfrak m$, because it coincides with the image of ${\rm ad}(A)$
for any regular $A \in \mathfrak t$. 
In particular, $m\cdot Y\in {\mathfrak t}^\perp \cap {\mathfrak m} $. 

Starting with an arbitrary $X\in {\mathfrak g}$, write $X = Y + Y^\prime + Z$, where $Y\in {\mathfrak m}, Y^\prime \in
{\mathfrak m}^\perp _{\mathfrak k}$
and $Z \in \mathfrak p$. Our goal is to present $X$ as a commutator. In doing so, we can replace $X$ by a conjugate element.
By Corollary~\ref{Cartan} we may assume that $Z \in {\mathfrak a}^\perp_{\mathfrak p}$. Now choose $m\in M$ for the given $Y\in{\mathfrak m}$
as above. Then $m\cdot Y \in {\mathfrak t}^\perp _{\mathfrak k} $ 
and $m\cdot Y^\prime \in {\mathfrak m}^\perp _{\mathfrak k}$. 
Since $m$ preserves ${\mathfrak a}$ and 
therefore also ${\mathfrak a}^\perp _{\mathfrak p}$, 
we have 
$m\cdot Z \in {\mathfrak a}^\perp_{\mathfrak p}$. In the decomposition
$m\cdot X = m\cdot Y + m\cdot Y^\prime + m\cdot Z $
the first and the second summand are in $\mathfrak k$ and in ${\mathfrak t}^\perp$,
whereas the last one is in $\mathfrak p$ 
and
in ${\mathfrak a}^\perp$. Since $\mathfrak k$ and $\mathfrak p$ are orthogonal with respect to the Killing form of $\mathfrak g$, it follows that
$m\cdot X \in {\mathfrak t}^\perp \cap {\mathfrak a}^\perp = ({\mathfrak t}+{\mathfrak a})^\perp$.

Finally, $\mathfrak t + \mathfrak a$ is a Cartan subalgebra in $\mathfrak g$.
Therefore $m\cdot X \in {\rm Im}\,{\rm ad} (A)$ for any regular $A \in {\mathfrak t}+ {\mathfrak a}$ by Lemma~\ref{ortho}.
\end{proof}

Consider the root system $\Delta $ of ${\mathfrak g}^{\mathbb C}$ with respect to the Cartan subalgebra ${\mathfrak h} = 
({\mathfrak t} + {\mathfrak a})^{\mathbb C}$. The roots vanishing on ${\mathfrak a}$ are called compact,
the remaining roots are called noncompact. The set of compact, resp. noncompact, roots is denoted by $\Delta _c$, resp. $\Delta _{nc}$.
Let $\theta : {\mathfrak g}^{\mathbb C} \to {\mathfrak g}^{\mathbb C}$ be the involutive automorphism acting as ${\rm Id}$ on
${\mathfrak k}^{\mathbb C}$ and as $-{\rm Id}$ on ${\mathfrak p}^{\mathbb C}$. One can choose an ordering of $\Delta $ in such a way
that the root $\theta ^T(\alpha )$ is negative if $\alpha \in \Delta _{nc}$ is positive. I.Satake ~\cite{Sa} showed then
that for every simple root $\alpha \in \Delta_{nc}$ there is another simple root $\alpha ^\prime \in \Delta _{nc}$, such that 
$$\theta ^T(\alpha ) = - \alpha ^\prime - \sum\ c_{\alpha \beta}\beta,$$ 
where $\beta $ ranges over simple roots in $\Delta _c$ and $c_{\alpha \beta }$ are non-negative integers.
The mapping $\alpha \mapsto \alpha ^\prime $ is involutive.
The Satake diagram of $\mathfrak g$ is defined as the Dynkin diagram of $\mathfrak g^{\mathbb C}$,
on which the compact roots are denoted by black circles, the non-compact roots are denoted by
white circles, and the white circles corresponding to $\alpha $ and $\alpha ^\prime \ne \alpha $
are joined by two-pointed arrows.

\begin{theorem}\label{arrows}
Let ${\mathfrak z}({\mathfrak m})$ be the center of ${\mathfrak m}$. Then ${\rm dim}\,{\mathfrak z}({\mathfrak m}) $ equals the number of two-pointed arrows on the Satake diagram.
\end{theorem}
\begin{proof} 
Let $r_c$ and $r_{nc}$ be the numbers of compact and non-compact simple roots,
respectively. Write $l = r_c + r_{nc}$ for the rank of ${\mathfrak g }^{\mathbb C}$.
Clearly,
$${\mathfrak m}^{\mathbb C} = {\mathfrak t}^{\mathbb C} + \sum _{\alpha \in \Delta _c}\ ({\mathfrak g}^{\mathbb C})_{\alpha} .$$
Let $\alpha_1 ,\ldots,\alpha_c$ be all compact simple roots. Then
$$\{H\in \mathfrak h \
\vert \
\alpha _1(H) = \ldots = \alpha _c(H) = 0\} = {\mathfrak z}({\mathfrak m})^{\mathbb C} + {\mathfrak a }^{\mathbb C}.$$
Therefore
$$l - r_c = {\rm dim}\,{\mathfrak z}\,({\mathfrak m}) + {\rm dim}\,{\mathfrak a}.$$
On the other hand, ${\rm dim}\,{\mathfrak a}$ is the number of simple restricted roots, i.e., $r_{nc}$ minus
the number of two-pointed arrows.
\end{proof}

\medskip
\noindent
{\it Proof of 
Theorem~\ref{simple}.}
All Satake diagrams of real simple Lie algebras are listed, 
see e.g. \cite{On}, Table 5. The diagrams without two-pointed arrows are easily found.
\hfill{$\square $}
\medskip
 
Finally, Theorem~\ref{inner} follows from a more general result, which we now prove.

\begin{theorem}
Let $\mathfrak g$ be a real semisimple Lie algebra of inner type. Then any element of $[{\mathfrak k}, {\mathfrak k}]+ {\mathfrak p}$
is the commutator of two elements of $\mathfrak g$. 
\end{theorem}
\begin{proof}
Let $\mathfrak c$ be the center of $\mathfrak k$, so that
$\mathfrak k = \mathfrak c + [\mathfrak k, \mathfrak k]$. Let $\mathfrak t $
be a Cartan subalgebra in $[\mathfrak k, \mathfrak k]$. 
Take $X \in [\mathfrak k, \mathfrak k]$ and $Y \in {\mathfrak p}$. Applying Corollary ~\ref {Cartan}
as in the proof of Theorem~\ref {main1}, we can find $k \in K$, so that $k\cdot X$
is orthogonal to $\mathfrak t$ with respect to the Killing form of the semisimple algebra
$[\mathfrak k, \mathfrak k]$. But the orthogonal complement to $\mathfrak t$ in $[\mathfrak k, \mathfrak k]$
is the same for all invariant non-degenerate bilinear forms on $[\mathfrak k, \mathfrak k]$
and, in particular, for the restriction of the Killing form of $\mathfrak g$. It follows that $k\cdot X \in {\mathfrak t}^\perp$.
On the other hand, $k\cdot X \in {\mathfrak c}^\perp$, hence $k\cdot X \in ({\mathfrak c + \mathfrak t})^\perp $.
Since we also have $k \cdot Y \in {\mathfrak p} = {\mathfrak k}^\perp \subset ({\mathfrak c} + {\mathfrak t})^\perp$,
it follows from Lemma~\ref{ortho}
that $k\cdot (X + Y) \in {\rm Im}\,{\rm ad}(A)$ for any regular $A$ in the Cartan subalgebra $\mathfrak c + \mathfrak t \subset
\mathfrak g$.

\end{proof} 

\end {document}